\documentclass[preprint]{imsart}

\pdfoutput=1
\RequirePackage[OT1]{fontenc}
\usepackage{amsthm, amsmath, natbib, amssymb, enumitem}
\RequirePackage[colorlinks,citecolor=blue,urlcolor=blue]{hyperref}

\usepackage{mathtools} 
\mathtoolsset{showonlyrefs}

\startlocaldefs
\numberwithin{equation}{section}
\theoremstyle{plain}
\newtheorem{theorem}{Theorem}[section]
\newtheorem{lemma}{Lemma}[section]

\newtheorem{proposition}{Proposition}[section]
\theoremstyle{remark}

\def\citeapos#1{\citeauthor{#1}'s (\citeyear{#1})}
\newcommand{\GB}{{ \mathrm{S} }}
\newcommand{\JSJS}{{ \mathrm{JS} }}
\newcommand{\RD}{{ \mathrm{Rdiff}_0 }}
\newcommand{\EE}{\mathrm{E}}
\newcommand{\rd}{\mathrm{d}}

\NewDocumentEnvironment{manual}{O{theorem}m}
 {%
  \addtocounter{theorem}{-1}%
  \renewcommand\thetheorem{#2}%
  \begin{#1}
 }
 {\end{#1}}

\endlocaldefs
\usepackage[margin=27truemm]{geometry}
\usepackage{booktabs}
\begin{document}
\begin{frontmatter}
\title{A new perspective on dominating \\ the James--Stein estimator}
\runtitle{Dominating the James--Stein estimator}

\begin{aug}
\author{\fnms{Yuzo} \snm{Maruyama}\thanksref{addr1,t1}\ead[label=e1]{maruyama@math.s.chiba-u.ac.jp}}
\and
\author{\fnms{Akimichi} \snm{Takemura}\thanksref{addr2,t2}\ead[label=e2]{a-takemura@biwako.shiga-u.ac.jp}}

\runauthor{Y.~Maruyama and A.~Takemura}

\address[addr1]{Chiba University \printead{e1} 
}

\address[addr2]{Shiga University \printead{e2}
}

\thankstext{t1}{supported by JSPS KAKENHI Grant Number 22K11933}
\thankstext{t2}{supported by JSPS KAKENHI Grant Number 24K14852}
\end{aug}

\begin{abstract}
This paper presents a novel approach to constructing estimators that dominate the classical James--Stein estimator under the quadratic loss for multivariate normal means. 
Building on Stein's risk representation, 
we introduce a new sufficient condition involving a monotonicity property of a transformed shrinkage function. 
We derive a general class of shrinkage estimators that satisfy minimaxity and dominance over the James--Stein estimator, including cases with polynomial or logarithmic convergence to the optimal shrinkage factor. 
We also provide conditions for uniform dominance across dimensions and for improved asymptotic risk performance. 
We present several examples and numerical validations to illustrate the theoretical results.
\end{abstract}

\begin{keyword}[class=MSC]
\kwd[Primary ]{62C20}
\end{keyword}

\begin{keyword}
\kwd{minimaxity}
\kwd{James--Stein estimator}
\end{keyword}

\end{frontmatter}

\section{Introduction and a main result}
\label{sec:intro}
Let $ X$ have a $p$-variate normal distribution 
$ \mathcal{N}_{p} (\theta, I_{p}) $. 
We consider the problem of estimating the mean vector $\theta$ under 
the quadratic loss function $\| \hat{\theta} - \theta \|^2$.
%\begin{equation}
% L(\theta,\hat{\theta})=\| \hat{\theta} - \theta \|^2
%=\sum_{i=1}^{p}(\hat{\theta}_i - \theta_{i})^2.
%\end{equation}
Then the risk function of an estimator $ \hat{\theta}(X)$ is
$ R(\theta,\hat{\theta})=\EE_\theta\bigl[ \| \hat{\theta}(X) - \theta \|^2\bigr]$.
%\begin{equation*}
% R(\theta,\hat{\theta})=\EE\bigl[ \| \hat{\theta}(X) - \theta \|^2\bigr]
%=\int_{\mathbb{R}^{p}}\frac{\| \hat{\theta}(x) - \theta \|^{2}}{(2\pi)^{p/2}}
%\exp\left(-\frac{\| x - \theta \|^{2}}{2}\right)\rd x.  
%\end{equation*}
The usual unbiased estimator, $ X $, has constant risk $p$ and is
minimax for $p\in\mathbb{N}$. 
\cite{Stein-1956} showed that, for $p\geq 3$,
there exist estimators dominating the usual estimator $X$
among a class of estimators of the form,
\begin{equation}\label{theta.a.b}
\hat{\theta}_{a,b}=\left(1-\frac{b}{a+\|X\|^2}\right)X,
\end{equation}
for large $a\geq 0$ and small $b>0$.
\cite{James-Stein-1961}
found an explicit dominating procedure among the estimators in \eqref{theta.a.b},
with $a=0$ and $b=p-2$,
\begin{equation}\label{JS}
\hat{\theta}_{\JSJS}(X)=\left( 1- \frac{p-2}{\| X \|^2}\right)X,
\end{equation}  
called the James--Stein estimator. 
Further, as shown in \cite{Baranchik-1964}, the James--Stein estimator is inadmissible,
as the positive--part estimator
\begin{equation}\label{JSPP}
\hat{\theta}_{\JSJS}^+(X)=\max\left(0, 1- \frac{p-2}{\| X \|^2}\right)X
\end{equation}  
dominates $\hat{\theta}_{\JSJS}$.
For a class of general shrinkage estimators %$\hat{\theta}_{\phi}(X)$ given by \eqref{hat.theta.phi}, orthogonally equivariant
of the form 
\begin{equation}\label{hat.theta.phi}
 \hat{\theta}_{\phi}(X)= \left( 1- \frac{\phi(\| X \|^2)}{\| X \|^2}\right)X,
\end{equation}
\cite{Baranchik-1970} proposed a sufficient condition for minimaxity, 
\{\ref{ba.1} \& \ref{ba.2} \& \ref{ba.3}\} below,
\begin{enumerate}[label= \textbf{A.\arabic*}, leftmargin=*]
\item\label{ba.1} $\phi'(w)\geq 0$ for all $w\geq 0$,
%\item\label{ba.2} $0\leq \phi(w)\leq 2(p-2)$ for all $w\geq 0$.
\item\label{ba.2} $ \phi(w)\geq 0$ for all $w\geq 0$,
\item\label{ba.3} $ \phi(w)\leq 2(p-2)$ for all $w\geq 0$.
\end{enumerate}
%Note the corresponding $\phi(w)$ to $\hat{\theta}_{a,b}$ given by \eqref{theta.a.b} is 
%\begin{equation}\label{phi.a.b} 
% \phi_{a,b}(w)=\frac{bw}{a+w}=b\Bigl(1-\frac{a}{a+w}\Bigr).
%\end{equation}
Then, by \citeauthor{Baranchik-1970}'s condition, 
$\hat{\theta}_{a,b}$ given by \eqref{theta.a.b}
is minimax for 
%\begin{equation}
% a\geq 0\text{ and }0\leq b\leq 2(p-2).
%\end{equation}
$a\geq 0$ and $0\leq b\leq 2(p-2)$.
Further,
\cite{Stein-1974} expressed the risk of
$\hat{\theta}_{\phi}(X)$ as
\begin{equation}\label{stein.identity.1}
R(\theta,\hat{\theta}_\phi)
%=\tilde{R}(\|\theta\|^2,\hat{\theta}_\phi)% \EE\bigl[\|\hat{\theta}_\phi-\theta\|^2\bigr]=
 =\EE_\theta\bigl[\hat{R}_\phi(\|X\|^2)\bigr],
\end{equation}
where 
\begin{equation}\label{stein.identity.2}
 \hat{R}_\phi(w)=p+\frac{\phi(w)}{w}\left\{\phi(w)-2(p-2)\right\}-4\phi'(w).
\end{equation}
Hence the shrinkage factor $\phi(w)$ in \eqref{hat.theta.phi} with the inequality $\hat{R}_\phi(w)\leq p$ for all $w\geq 0$,
implies minimaxity of $\hat{\theta}_\phi$. 
We see that \{\ref{ba.1} \&  \ref{ba.2} \& \ref{ba.3}\} is a tractable sufficient condition 
for $\hat{R}_\phi(w)\leq p$ for all $w\geq 0$.

As we mentioned in \eqref{JSPP}, $\hat{\theta}_{\JSJS}(X)$
is inadmissible, which have given rise to theoretical challenges:
the problem of finding estimators dominating $\hat{\theta}_{\JSJS}(X)$.
\cite{Kubokawa-1991} showed that
the generalized Bayes estimator with respect to \citeapos{Stein-1974} prior $\|\theta\|^{2-p} $
dominates $\hat{\theta}_{\JSJS}(X)$,
\begin{equation}
 \hat{\theta}_{\GB}(X)=\frac{\int \theta \exp(-\|X-\theta\|^2/2)\|\theta\|^{2-p}\rd \theta}
{\int \exp(-\|X-\theta\|^2/2)\|\theta\|^{2-p}\rd \theta}
=\left( 1- \frac{\phi_{\GB}(\| X \|^2)}{\| X \|^2}\right)X,
\end{equation}
where
%\begin{equation}\label{phi.GB}
%\begin{split}
% \phi_{\GB}(w)
%&=w\frac{\int_0^\infty (g+1)^{-p/2-1} \exp(wg/\{2(g+1)\})\rd g}{\int_0^\infty (g+1)^{-p/2} \exp(wg/\{2(g+1)\})\rd g} \\
%&=p-2-\frac{2}{\int_0^\infty (g+1)^{-p/2} \exp(wg/\{2(g+1)\})\rd g}.
%\end{split} 
%\end{equation}
\begin{equation}\label{phi.GB}
 \phi_{\GB}(w)=p-2-\frac{2}{\int_0^\infty (g+1)^{-p/2} \exp(wg/\{2(g+1)\})\rd g}.
\end{equation}
\cite{Kubokawa-1994} gave a sufficient condition for dominating
$\hat{\theta}_{\JSJS}(X)$, \{\ref{ba.1} \& \ref{ku.1} \& \ref{ku.2}\},
where:
\begin{enumerate}[label= \textbf{A.\arabic*}, leftmargin=*, resume]
%[label= \textbf{.\arabic*}, leftmargin=*,regime]
\item\label{ku.1} $\lim_{w\to\infty} \phi(w)=p-2$, %and %$\phi'(w)\geq 0$ for all $w\geq 0$.
\item\label{ku.2} $\phi(w)\geq \phi_{\GB}(w)$ for all $w\geq 0$.
\end{enumerate}
As in \eqref{phi.GB}, $\phi_{\GB}(w)$ approaches $p-2$ 
at an exponential order as $w\to\infty$.
Hence, 
any $\phi(w)$ approaching $p-2$ at a polynomial order, including
\begin{equation}
  \phi_{a,p-2}(w)=\frac{(p-2)w}{a+w}=(p-2)\Bigl(1-\frac{a}{a+w}\Bigr),\quad \text{for}\quad 
a>0,
\end{equation}
does not satisfy \ref{ku.2} since $ \phi_{\GB}(w)>\phi_{a,p-2}(w)$ for large $w$.
Moreover, a general sufficient condition for dominating $\hat{\theta}_{\JSJS}(X)$
 by functions converging to $p-2$ at a polynomial rate 
has not been known until now.

In this paper, we take a new approach.
By \eqref{stein.identity.1} and \eqref{stein.identity.2},
the difference of the risks is 
\begin{equation}\label{risk.diff.1}
R(\theta,\hat{\theta}_{\JSJS})- R(\theta,\hat{\theta}_{\phi})
 =\EE_\theta\bigl[\hat{R}_{\JSJS}(\|X\|^2)-\hat{R}_\phi(\|X\|^2)\bigr]
 \end{equation}
where
\begin{equation}\label{Phi.1}
 \begin{split}
 \hat{R}_{\JSJS}(w)-\hat{R}_\phi(w)&=4\phi'(w)-\frac{\{p-2-\phi(w)\}^2}{w}\\
&=\frac{\{p-2-\phi(w)\}^2}{w}\left(4\frac{w\phi'(w)}{\{p-2-\phi(w)\}^2}-1\right).
\end{split}
\end{equation}
For our methodology, 
the monotonic non-decreasing property of the function $w\phi'(w)/\{p-2-\phi(w)\}^2$
which appears in \eqref{Phi.1}, and the dominance at the origin are central. 
Thus we assume 
\begin{enumerate}[label= \textbf{A.\arabic*}, leftmargin=*, resume]
%[label= \textbf{MT.\arabic*}, leftmargin=*]
\item\label{mt.1} $w\phi'(w)/\{p-2-\phi(w)\}^2$ is monotone non-decreasing,
\item\label{mt.2} $R(0,\hat{\theta}_{\JSJS})\geq R(0,\hat{\theta}_{\phi})$, 
i.e.~dominance at the origin,
\end{enumerate}
in addition to \{\ref{ba.1} \& \ref{ba.2} \& \ref{ku.1}\}.
%Assumption \ref{mt.2} is very mild, which only requires the domination at the origin.
We then obtain the main result of this paper.
\begin{theorem}\label{thm:main}
Assume \{\ref{ba.1} \& \ref{ba.2} \& \ref{ku.1} \& \ref{mt.1} \& \ref{mt.2}\}. Then
$\hat{\theta}_{\phi}$ %with \eqref{phi.Phi} 
dominates the James--Stein estimator 
$\hat{\theta}_{\JSJS}(X)$.
\end{theorem}
In order to gain insight into a function $\phi(w)$ satisfying \ref{mt.1},
let us instead suppose that a positive monotone non-decreasing $\Phi(w)$ is given. 
Then our $\phi(w)$ can be formulated 
as the solution of a separable differential equation 
\begin{equation}\label{sde}
 \frac{\rd}{\rd w}\left(\frac{1}{p-2-\phi(w)}\right)=\frac{\Phi(w)}{w},
\end{equation}
with the solution
\begin{equation}\label{phi.Phi}
 \phi(w)=p-2-\frac{1}{\int_0^w t^{-1}\Phi(t)\rd t+ C},
\end{equation}
where $C$ represents the constant of integration.
When $C\geq 1/(p-2)$,
%\begin{equation}\label{CCC}
% C\geq \frac{1}{p-2}
%\end{equation}
the function $\phi(w)$ given by \eqref{phi.Phi}
satisfies \{\ref{ba.1} \& \ref{ba.2} \& \ref{ku.1}\}.
Hence, as an alternative but equivalent expression,
Theorem \ref{thm:main} can be represented as follows.

\begin{manual}{\ref{thm:main}$^\prime$}
Suppose $\Phi(w)$ is positive and monotone non-decreasing. Then 
the estimator $\hat{\theta}_{\phi}=(1-\phi(\|X\|^2)/\|X\|^2)X$ where
\begin{equation}\label{phi.Phi.0}
\phi(w)=p-2-
 \frac{1}{\int_0^w t^{-1}\Phi(t)\rd t+ C}\quad\text{for}\quad C\geq \frac{1}{p-2}
\end{equation}
dominates the James--Stein estimator 
$\hat{\theta}_{\JSJS}(X)$ 
if \ref{mt.2} is satisfied.
\end{manual}
%\begin{theorem}[altenative]%\label{thm:main}
%Suppose $\Phi(w)$ is positive and monotone non-decreasing. Then 
%the estimator $\hat{\theta}_{\phi}=(1-\phi(\|X\|^2)/\|X\|^2)X$ where
%\begin{equation}\label{phi.Phi.0}
%\phi(w)=p-2-
% \frac{1}{\int_0^w t^{-1}\Phi(t)\rd t+ C}\quad\text{for}\quad C\geq \frac{1}{p-2}
%\end{equation}
%dominates the James--Stein estimator 
%$\hat{\theta}_{\JSJS}(X)$ 
%if \ref{mt.2} is satisfied.
%\end{theorem}

%In addition to \ref{mt.1}, we also assume
%\begin{enumerate}[label= \textbf{A.\arabic*}, leftmargin=*, resume]
%\item\label{mt.2} $R(0,\hat{\theta}_{\JSJS})\geq R(0,\hat{\theta}_{\phi})$.  
%\end{enumerate}
%
%Since $\Phi(w)$ is monotone non-decreasing, 
%$\phi(w;\Phi)$ is monotone
%$ \lim_{w\to\infty}\int_0^w t^{-1}\Phi(t)\rd t$ and $ \lim_{w\to\infty}\phi(w;\Phi)=p-2$
%follow.
%In this paper, we will present a variety of estimators that, while not meeting the sufficient conditions proposed by \cite{Kubokawa-1994}, do satisfy the requirements of Theorem \ref{thm:main}.

The organization of this paper is as follows.
We give the proof of Theorem \ref{thm:main}${}^\prime$ in Section \ref{sec:proof.1}.
In general, Assumption \ref{mt.2} must be checked individually for each dimension of the parameter.
To minimize the number of times Assumption \ref{mt.2} needs to be verified, 
we provide some sufficient conditions for the uniformity of Assumption \ref{mt.2} in Section \ref{sec:unif}.
Then, in Section \ref{sec:examples}, we demonstrate
estimators which satisfy the sufficient condition of 
Theorem \ref{thm:main} and the uniformity, including
\begin{equation}
% \Phi(w)=\frac{w}{p-2},\quad C=\frac{1}{p-2}, \quad\text{and } 
%\phi(w)=(p-2)\frac{w}{w+1},\quad
\hat{\theta}_{1,p-2}=\left(1-\frac{p-2}{\|X\|^2+1}\right)X,
\end{equation}
which is the member of \citeauthor{Stein-1956}'s initial class \eqref{theta.a.b}.
Notably, the corresponding $\phi(w)$, $p-2-(p-2)/(w+1)$,  
approaches $p-2$ at a polynomial order.
%As we point out in Remark \ref{rem:zenkin}, 
%\cite{Maruyama-Takemura-2024} investigated the difference in risks 
%between $\hat{\theta}_{\JSJS}$ and $\hat{\theta}_{\JSJS}^+$,
%\begin{equation}
%\lim_{\|\theta\|\to\infty}
%%\frac{\|\theta\|^{(p+1)/2}e^{\|\theta\|^2/2}}{e^{\|\theta\|\sqrt{p-2}}} 
%e^{\|\theta\|^2/3}
% \bigl\{R(\theta,\hat{\theta}_{\JSJS})-R(\theta,\hat{\theta}_{\JSJS}^+)\bigr\}
%=0,
%%=4\frac{(p-2)^{(p-1)/4}}{\sqrt{2\pi}\exp(p/2-1)}.
%\end{equation}
%which implies that the asymptotic risk gain by the positive--part estimator is quite small.
%We see that with the change $\|x\|^2\to \|x\|^2+1$ in the denominator
%of the James--Stein estimator, 
%we have an improved asymptotic risk gain
%\begin{equation}
%\lim_{\|\theta\|\to\infty}
%%\frac{\|\theta\|^{(p+1)/2}e^{\|\theta\|^2/2}}{e^{\|\theta\|\sqrt{p-2}}} 
%\|\theta\|^4
% \bigl\{R(\theta,\hat{\theta}_{\JSJS})-R(\theta,\hat{\theta}_{1,p-2})\bigr\}
%=4(p-2).
%%=4\frac{(p-2)^{(p-1)/4}}{\sqrt{2\pi}\exp(p/2-1)}.
%\end{equation}
%From this viewpoint, in Section \ref{sec:examples}, 
We also propose 
%a further improvement
\begin{equation}
\left(1-\frac{1}{\|X\|^2}\left\{p-2-\frac{2}{\log (\|X\|^2+1)+2/(p-2)}\right\} \right)
X
\end{equation}
which dominates the James--Stein estimator with $\phi$ 
approaching $p-2$ at a logarithmic order.
In Section \ref{rem:zenkin}, 
we investigate the extent of asymptotic improvement achieved 
by shrinkage estimators for $\phi$ that converge to $p-2$ at polynomial or logarithmic rates, 
as $\|\theta\| \to \infty$.
%
%the better risk gain
%\begin{equation}
%\lim_{\|\theta\|\to\infty}
%\|\theta\|^2\{\log \|\theta\|^2\}^{2} 
%\bigl\{R(\theta,\hat{\theta}_{\JSJS})-R(\theta,\hat{\theta})\bigr\}
%=4.
%\end{equation}
Section \ref{sec:CR} gives some concluding remarks.
%For all estimators proposed in Section \ref{sec:proof}, 
%$1-\phi(w)/w$ takes negative values
%and hence their positive--part estimators dominate it.
%With our methodology, we propose a genuine positive--part estimator

%\begin{equation}
%R(\theta,\hat{\theta}_{\JSJS})- R(\theta,\hat{\theta}_{\phi})\geq 0
%\end{equation}
%for all $\theta\in\mathbb{R}^p$ 
%Recall the risk of $\hat{\theta}_{\JSJS}$ at $0\in\mathbb{R}^p$ is equal to $2$
%since
%\begin{equation}
%R(0,\hat{\theta}_{\JSJS})
%%=\tilde{R}(\|\theta\|^2,\hat{\theta}_\phi)% \EE\bigl[\|\hat{\theta}_\phi-\theta\|^2\bigr]=
%=\EE_0\bigl[\hat{R}_{\JSJS}\bigr]=p-(p-2)^2\EE_0[1/\|X\|^2]=p-(p-2)=2
%\end{equation}
%Then the sufficient condition \eqref{sufficient.1} 
%\begin{equation}
% R(0,\hat{\theta}_{\phi})=\EE\bigl[(1-\phi(\|X\|^2)/\|X\|^2)^2\|X\|^2\bigr]
%\end{equation}
%
%\section{Proof of Theorem \ref{thm:main} and the uniformity in dimensionality}% and uniformity on $p$}
%\label{sec:proof}
\section{Proof of Theorem \ref{thm:main}${}^\prime$}
\label{sec:proof.1}
%As in \eqref{sde},
We suppose a positive monotone non-decreasing function $\Phi(w)$ is first given
and the induced $\phi(w)$ is 
\begin{equation}%\label{phi.Phi.0}
 \phi(w)=p-2-\frac{1}{\int_0^w t^{-1}\Phi(t)\rd t+ C},
\end{equation}
with $C\geq 1/(p-2)$, as in \eqref{phi.Phi.0}. 
%Then, by Theorem \ref{thm:main}${}^*$, it suffices to show that 
%$R(\theta,\hat{\theta}_{\JSJS})- R(\theta,\hat{\theta}_{\phi})\geq 0$
%under Assumption \ref{mt.2}.

Let $P_\lambda(k)$ and $g_{p+2k}(w)$ be the probability mass function of 
Poisson distribution with mean parameter $\lambda/2=\|\theta\|^2/2$ and
the probability density function of Chi-squared distribution 
with $p+2k$ degrees of freedom, respectively.
Then, by \eqref{risk.diff.1} and \eqref{Phi.1},
we have
%\begin{equation}\label{risk.diff.1}
%R(\theta,\hat{\theta}_{\JSJS})- R(\theta,\hat{\theta}_{\phi})
% =\EE_\theta\bigl[\hat{R}_{JS}(\|X\|^2)-\hat{R}_\phi(\|X\|^2)\bigr]
% \end{equation}
%\begin{equation}\label{Phi.1}
% \begin{split}
% \hat{R}_{\JSJS}(w)-\hat{R}_\phi(w)&=4\phi'(w)-\frac{\{p-2-\phi(w)\}^2}{w}\\
%&=\frac{\{p-2-\phi(w)\}^2}{w}\left(4\frac{w\phi'(w)}{\{p-2-\phi(w)\}^2}-1\right).
%\end{split}
%\end{equation}
\begin{equation}\label{E-RJS-R-G}
R(\theta,\hat{\theta}_{\JSJS})- R(\theta,\hat{\theta}_{\phi})
=\sum_{k=0}^\infty 
P_\lambda(k)\int_0^\infty \left\{4\Phi(w)-1\right\}h(w) g_{p+2k}(w)\rd w,
\end{equation}
where
\begin{equation}\label{varphi.1}
 h(w)=\frac{\{p-2-\phi(w)\}^2}{w}=
\frac{1}{w\{\int_0^w t^{-1}\Phi(t)\rd t+ C\}^2}.
\end{equation}
In each term of \eqref{E-RJS-R-G}, we have
\begin{equation}
\begin{split}
 \int_0^\infty \left\{4\Phi(w)-1\right\}h(w) g_{p+2k}(w)\rd w
&=\int_0^\infty h(w) g_{p+2k}(w)\rd w
\left(4\frac{\int_0^\infty \Phi(w) h(w) g_{p+2k}(w)\rd w}
{\int_0^\infty h(w) g_{p+2k}(w)\rd w}-1\right)\label{E-RJS-R-G.1}\\
&=\int_0^\infty h(w) g_{p+2k}(w)\rd w
\left(4\frac{\int_0^\infty \Phi(w) h(w) w^{p/2+k-1}e^{-w/2}\rd w}
{\int_0^\infty h(w) w^{p/2+k-1}e^{-w/2}\rd w}-1\right).
%&=\int_0^\infty w^{-1}\{c(w)\}^2 g_{p+2k}(w)\rd w
%\left(\frac{\int_0^\infty w^k h(w;p)\rd w}{\int_0^\infty (w+1)^{-1}w^k h(w;p)\rd w}
%-2\right)
\end{split} 
\end{equation}
Since $\Phi(w)$ is monotone non-decreasing, by the covariance inequality,
we have
\begin{equation}\label{cov.ineq}
% \begin{split}
\frac{\int_0^\infty \Phi(w) h(w) w^{p/2+k-1}e^{-w/2}\rd w}
{\int_0^\infty h(w) w^{p/2+k-1}e^{-w/2}\rd w}
\geq 
\frac{\int_0^\infty \Phi(w) h(w) w^{p/2-1}e^{-w/2}\rd w}
{\int_0^\infty h(w) w^{p/2-1}e^{-w/2}\rd w},
%\\
%&\geq 
%\frac{\int_0^\infty \Phi(w) h(w) g_p(w)\rd w}
%{\int_0^\infty h(w) g_p(w)\rd w}
%\end{split}
\end{equation}
for $k=0,1,\dots$. It follows from \eqref{E-RJS-R-G.1} and \eqref{cov.ineq} that
\begin{equation}\label{E-RJS-R-G.2}
% \begin{split}
%& \int_0^\infty \left\{4\Phi(w)-1\right\}h(w) g_{p+2k}(w)\rd w\\
%&\geq   
%\int_0^\infty h(w) g_{p+2k}(w)\rd w
%\left(4\frac{\int_0^\infty \Phi(w) h(w) w^{p/2-1}e^{-w/2}\rd w}
%{\int_0^\infty h(w) w^{p/2-1}e^{-w/2}\rd w}-1\right).
% \end{split}
 \int_0^\infty \left\{4\Phi(w)-1\right\}h(w) g_{p+2k}(w)\rd w
\geq   
\int_0^\infty h(w) g_{p+2k}(w)\rd w
\left(4\frac{\int_0^\infty \Phi(w) h(w) w^{p/2-1}e^{-w/2}\rd w}
{\int_0^\infty h(w) w^{p/2-1}e^{-w/2}\rd w}-1\right).
\end{equation}
Note, by \eqref{E-RJS-R-G},
\begin{equation}\label{genten}
\begin{split}
 R(0,\hat{\theta}_{\JSJS})- R(0,\hat{\theta}_{\phi})
&=\int_0^\infty \left\{4\Phi(w)-1\right\}h(w) g_p(w)\rd w\\
&=\int_0^\infty h(w) g_p(w)\rd w
\left(4\frac{\int_0^\infty \Phi(w) h(w) w^{p/2-1}e^{-w/2}\rd w}
{\int_0^\infty h(w) w^{p/2-1}e^{-w/2}\rd w}-1\right).
\end{split}
\end{equation}
Hence, by \eqref{E-RJS-R-G}, \eqref{E-RJS-R-G.2} and \eqref{genten}, we have
\begin{equation}
% \begin{split}
%&R(\theta,\hat{\theta}_{\JSJS})- R(\theta,\hat{\theta}_{\phi})\\
%&\geq \left\{\sum_{k=0}^\infty 
%P_\lambda(k)
%\frac{\int_0^\infty h(w) g_{p+2k}(w)\rd w}{\int_0^\infty h(w) g_{p}(w)\rd w}\right\}
%\left\{R(0,\hat{\theta}_{\JSJS})- R(0,\hat{\theta}_{\phi})\right\}
%\end{split}
R(\theta,\hat{\theta}_{\JSJS})- R(\theta,\hat{\theta}_{\phi})
\geq \left\{\sum_{k=0}^\infty 
P_\lambda(k)
\frac{\int_0^\infty h(w) g_{p+2k}(w)\rd w}{\int_0^\infty h(w) g_{p}(w)\rd w}\right\}
\left\{R(0,\hat{\theta}_{\JSJS})- R(0,\hat{\theta}_{\phi})\right\}
\end{equation}
which is non-negative under Assumption \ref{mt.2}, completing the proof.

\section{Uniformity in dimensionality}
\label{sec:unif}
%In this section, 
In general, Assumption \ref{mt.2}, 
$R(0,\hat{\theta}_{\JSJS})\geq R(0,\hat{\theta}_{\phi})$,
must be checked individually for each dimension of the parameter. 
To minimize the number of times Assumption \ref{mt.2} needs to be verified, 
in this section,
we provide the sufficient conditions for the uniformity of Assumption \ref{mt.2}, 
with respect to the dimensionality $p$.

Let
\begin{equation}\label{IPhi}
%\begin{split}
% I(\Phi,C;p)&=\frac{\int_0^\infty \Phi(w) h(w) w^{p/2-1}e^{-w/2}\rd w}
%{\int_0^\infty h(w) w^{p/2-1}e^{-w/2}\rd w}\\
%&=\frac{\displaystyle \int_0^\infty \Phi(w) \frac{w^{p/2-1}e^{-w/2}}{w\{\int_0^w t^{-1}\Phi(t)\rd t+ C\}^2} \rd w}
%{\displaystyle \int_0^\infty \frac{w^{p/2-1}e^{-w/2}}{w\{\int_0^w t^{-1}\Phi(t)\rd t+ C\}^2} \rd w},
%\end{split} 
 I(\Phi,C;p)=\frac{\int_0^\infty \Phi(w) h(w) w^{p/2-1}e^{-w/2}\rd w}
{\int_0^\infty h(w) w^{p/2-1}e^{-w/2}\rd w}
=\frac{\displaystyle \int_0^\infty \Phi(w) \frac{w^{p/2-1}e^{-w/2}}{w\{\int_0^w t^{-1}\Phi(t)\rd t+ C\}^2} \rd w}
{\displaystyle \int_0^\infty \frac{w^{p/2-1}e^{-w/2}}{w\{\int_0^w t^{-1}\Phi(t)\rd t+ C\}^2} \rd w},
\end{equation}
which appears in \eqref{E-RJS-R-G.2}.
%As in \eqref{CCC}, recall the assumption $C\geq 1/(p-2)$
%so that the function $\phi(w)$ given by \eqref{phi.Phi}
%satisfies \{\ref{ba.1}, \ref{ba.2}, \ref{ku.1}\}.
%where, as in \eqref{varphi.1}
%\begin{equation}
% h(w)=\frac{\{p-2-\phi(w)\}^2}{w}=
%\frac{1}{w\{\int_0^w t^{-1}\Phi(t)\rd t+ C\}^2}.
%\end{equation}
%As in $I(\Phi,C;p)$ given by \eqref{IPhi},
%we clearly state the dimensionality of the parameter $\theta\in\mathbb{R}^p$. 
We also clearly state the dimensionality of the parameter $\theta\in\mathbb{R}^p$ for
the difference in risk at $\theta=0$,
\begin{equation}\label{eq:RD}
 \RD(\hat{\theta}_{\JSJS},\hat{\theta}_{\phi};p)=R(0,\hat{\theta}_{\JSJS})- R(0,\hat{\theta}_{\phi}).
\end{equation}
Then, by \eqref{genten}, \eqref{IPhi} and \eqref{eq:RD}, the difference in risk at $\theta=0$ is represented as
\begin{equation}\label{RDRD.00}
 \RD(\hat{\theta}_{\JSJS},\hat{\theta}_{\phi};p)
=\{4I(\Phi,C;p)-1\}\int_0^\infty h(w) g_p(w)\rd w.
%\left(4\frac{\int_0^\infty \Phi(w) h(w) w^{p/2-1}e^{-w/2}\rd w}
%{\int_0^\infty h(w) w^{p/2-1}e^{-w/2}\rd w}-1\right)
\end{equation}
Suppose both $ \Phi(w) $ and $C$ do not depend on $p$. 
Then, by the covariance inequality, we have
\begin{equation}
 I(\Phi,C;p+1)\geq I(\Phi,C;p)
\end{equation}
for any $p+1>p \geq 3$ and 
\begin{equation}\label{hikaku.0}
\begin{split}
  \RD(\hat{\theta}_{\JSJS},\hat{\theta}_{\phi};p+1)
%&=\int_0^\infty h(w) g_{p+1}(w)\rd w
%\left(4\frac{\int_0^\infty \Phi(w)h(w) g_{p+1}(w)\rd w}{\int_0^\infty h(w) g_{p+1}(w)\rd w}-1\right)\\
&=\{4I(\Phi,C;p+1)-1\}\int_0^\infty h(w) g_{p+1}(w)\rd w\\
&\geq \{4I(\Phi,C;p)-1\}\int_0^\infty h(w) g_{p+1}(w)\rd w\\
%\left(4\frac{\int_0^\infty \Phi(w)h(w) w^{(p+1)/2-1}e^{-w/2}\rd w}
%{\int_0^\infty h(w) w^{(p+1)/2-1}e^{-w/2}\rd w}-1\right) \label{uniformity.1}\\
%&\geq\int_0^\infty h(w) g_{p+1}(w)\rd w
%\left(4\frac{\int_0^\infty \Phi(w)h(w) w^{p/2-1}e^{-w/2}\rd w}
%{\int_0^\infty h(w) w^{p/2-1}e^{-w/2}\rd w}-1\right)\notag \\
%&= \int_0^\infty h(w) g_{p+1}(w)\rd w
%\left(4\frac{\int_0^\infty \Phi(w)h(w) g_p(w)\rd w}{\int_0^\infty h(w) g_p(w)\rd w}-1\right)\\
&=\frac{\int_0^\infty h(w) g_{p+1}(w)\rd w}
{\int_0^\infty h(w) g_p(w)\rd w}\RD(\hat{\theta}_{\JSJS},\hat{\theta}_{\phi};p).
\end{split} 
\end{equation}
Then we have the following result.
\begin{proposition}\label{prop:1}
Suppose that both $ \Phi(w) $ and $C$ do not depend on $p$ and that
\begin{equation}
 \RD(\hat{\theta}_{\JSJS},\hat{\theta}_{\phi};p_*)\geq 0 \ \text{for some} \ p_*\geq 3.
\end{equation}
Then 
\begin{equation}
 \RD(\hat{\theta}_{\JSJS},\hat{\theta}_{\phi};p)\geq 0 \ \text{for all} \ p>p_*.
\end{equation}
\end{proposition}
%\begin{example}
%Let $\Phi(w)=w^2 $ and $C=1$. Then
%\end{example}
Proposition \ref{prop:1}, 
along with the other two Propositions presented in this section, 
will be used in the following section to demonstrate that the estimators under consideration 
provide a uniform improvement over the James--Stein estimator 
with respect to the parameter dimension $p$.
       
Recall $\phi(0)=p-2-1/C$ and hence $\phi(w)\in(p-2-1/C,p-2)$ for all $w\geq 0$.
The function $\phi(w)$ with large $p$ and with $C$ independent of $p$, 
tends to take values only around $p-2$.
Based on this consideration, it is natural to set $C=1/(p-2)$,
which ensures that $\phi(0)=0$. Under $C=1/(p-2)$, we have the following lemma.
\begin{lemma}\label{lem:C.p-2}
Let $C=1/(p-2)$. Suppose that $ \Phi(w) $ does not depend on $p$ and that
there exists a positive $\beta$ such that
\begin{equation}\label{eq:beta.0}
 \frac{\Phi(w)}{\int_0^w t^{-1}\Phi(t)\rd t }\leq \beta \text{ for all }w\geq 0.
\end{equation}
Then, for $p\geq \beta+2$,
\begin{equation}
  I\Bigl(\Phi,\frac{1}{p-1};p+1\Bigr)\geq   I\Bigl(\Phi,\frac{1}{p-2};p\Bigr).
\end{equation}
\end{lemma}

\begin{proof}
By \eqref{IPhi}, we have
\begin{equation}
%\begin{split}
 I\Bigl(\Phi,\frac{1}{p-1};p+1\Bigr)
%\frac{\int_0^\infty \Phi(w)h(w) w^{(p+1)/2-1}e^{-w/2}\rd w}
%{\int_0^\infty h(w) w^{(p+1)/2-1}e^{-w/2}\rd w}\\
%&=\frac{\displaystyle\int_0^\infty \Phi(w)\frac{w^{(p+1)/2-1}e^{-w/2}}{w\{\int_0^w t^{-1}\Phi(t)\rd t+ 1/(\{p+1\}-2)\}^2} \rd w}{\displaystyle\int_0^\infty \Phi(w)\frac{w^{(p+1)/2-1}e^{-w/2}}{w\{\int_0^w t^{-1}\Phi(t)\rd t+ 1/(\{p+1\}-2)\}^2} \rd w}\\
=\frac{\displaystyle\int_0^\infty \Phi(w)\frac{f(w;\Phi,p)w^{p/2-1}e^{-w/2}}
{w\{\int_0^w t^{-1}\Phi(t)\rd t+ 1/(p-2)\}^2} \rd w}
{\displaystyle\int_0^\infty \frac{f(w;\Phi,p)w^{p/2-1}e^{-w/2}}
{w\{\int_0^w t^{-1}\Phi(t)\rd t+ 1/(p-2)\}^2} \rd w},
\end{equation}
where
\begin{equation}\label{eq:f}
 f(w;\Phi,p)=w^{1/2}
\left(\frac{\int_0^w t^{-1} \Phi(t) \rd t+1/(p-2)}
{\int_0^w t^{-1} \Phi(t) \rd t+1/(p-1)}\right)^2.
\end{equation}
Recall $ \Phi(w)$ is monotone non-decreasing in $w$. 
If $ f(w;\Phi,p)$ is monotone increasing in $w$, the covariance inequality 
implies that
\begin{equation}
  I\Bigl(\Phi,\frac{1}{p-1};p+1\Bigr)
\geq 
\frac{\displaystyle\int_0^\infty \frac{\Phi(w)w^{p/2-1}e^{-w/2}\rd w}
{w\{\int_0^w t^{-1}\Phi(t)\rd t+ 1/(p-2)\}^2} }
{\displaystyle\int_0^\infty \frac{w^{p/2-1}e^{-w/2}\rd w}
{w\{\int_0^w t^{-1}\Phi(t)\rd t+ 1/(p-2)\}^2} }
=I\Bigl(\Phi,\frac{1}{p-2};p\Bigr),
\end{equation}
and the result follows. Hence it suffices to show 
the function $f(w;\Phi,p)$ given by \eqref{eq:f}
is increasing in $w$ for $p\geq \beta+2$.
Note 
\begin{equation}
 \log f(w)=\frac{1}{2}\log w+ 2\log
\frac{\int_0^w t^{-1} \Phi(t) \rd t+1/(p-2)}{\int_0^w t^{-1} \Phi(t) \rd t+1/(p-1)}.
\end{equation} 
Then we have
\begin{equation}
\begin{split}
 \frac{\rd}{\rd w}\log f(w)
&=\frac{1}{2}\frac{1}{w}+ 2
\frac{\Phi(w)/w}{\int_0^w t^{-1} \Phi(t) \rd t+1/(p-2)}-
2\frac{\Phi(w)/w}{\int_0^w t^{-1} \Phi(t) \rd t+1/(p-1)}\\
%&=\frac{1}{2}\frac{1}{w}+ 2\frac{\Phi(w)}{w}
%\frac{1/(p-1)-1/(p-2)}{\{\int_0^w t^{-1} \Phi(t) \rd t+1/(p-2)\}\{\int_0^w t^{-1} \Phi(t) \rd t+1/(p-1)\}}\\
&=\frac{1}{2w}
\left(1- \frac{4\Phi(w)}{\{(p-2)\int_0^w t^{-1} \Phi(t) \rd t+1\}
\{(p-1)\int_0^w t^{-1} \Phi(t) \rd t+1\}}\right)\\
&\geq \frac{1}{2w}
\left(1- \frac{4(p-2)\int_0^w t^{-1} \Phi(t) \rd t}{\{(p-2)\int_0^w t^{-1} \Phi(t) \rd t+1\}
\{(p-1)\int_0^w t^{-1} \Phi(t) \rd t+1\}}\right),
\end{split} 
\end{equation}
where the inequality is due to \eqref{eq:beta.0} and $\beta\leq p-2$.
Let $ y=\int_0^w t^{-1} \Phi(t) \rd t$. Then
\begin{equation}
%\begin{split}
%& 1- 
%\frac{4(p-2)y}{\{(p-2)y+1\}\{(p-1)y+1\}}\\
%& \geq 1- \frac{4(p-2)y}{\{(p-2)y+1\}^2}
%=\frac{\{(p-2)y-1\}^2}{\{(p-2)y+1\}^2}
%\geq 0,
%\end{split} 
 1- 
\frac{4(p-2)y}{\{(p-2)y+1\}\{(p-1)y+1\}}
 \geq 1- \frac{4(p-2)y}{\{(p-2)y+1\}^2}
=\frac{\{(p-2)y-1\}^2}{\{(p-2)y+1\}^2}
\geq 0,
\end{equation}
which completes the proof.
\end{proof}
As in \eqref{hikaku.0}, with Lemma \ref{lem:C.p-2}, we have
\begin{equation}\label{hikaku}
\begin{split}
  \RD(\hat{\theta}_{\JSJS},\hat{\theta}_{\phi};p+1)
&\geq \frac{\int_0^\infty h(w) g_{p+1}(w)\rd w}
{\int_0^\infty h(w) g_p(w)\rd w}\RD(\hat{\theta}_{\JSJS},\hat{\theta}_{\phi};p),
\end{split} 
\end{equation}
for $ p\geq \beta+2$.
Then 
we have the following result.
\begin{proposition}\label{prop:2}
Let $C=1/(p-2)$. Suppose that $ \Phi(w) $ does not depend on $p$ and that
there exists a positive $\beta$ such that
\begin{equation}\label{eq:beta.00}
 \frac{\Phi(w)}{\int_0^w t^{-1}\Phi(t)\rd t }\leq \beta \text{ for all }w\geq 0.
\end{equation}
Further suppose that $\RD(\hat{\theta}_{\JSJS},\hat{\theta}_{\phi};p_*)\geq 0  \text{ for some } p_*\geq \beta+2$.
%\begin{equation}
% \RD(\hat{\theta}_{\JSJS},\hat{\theta}_{\phi};p_*)\geq 0  \text{ for some } p_*\geq \beta+2.
%\end{equation}
Then 
\begin{equation}
 \RD(\hat{\theta}_{\JSJS},\hat{\theta}_{\phi};p)\geq 0  \text{ for all }  p>p_*.
\end{equation}
\end{proposition}

\medskip

Finally suppose that $ \Phi(w)$ also depends on $p$ (say $\Phi_p(w)$) and 
that $C=1/(p-2)$ as in Proposition \ref{prop:2}. 
Further suppose 
%general results such as Propositions \ref{prop:1} and \ref{prop:2} cannot be obtained.
that $I(\Phi_p,1/(p-2);p)$ is bounded below by $\tilde{I}(p)$.
%which is monotone increasing in $p$. 
Then we have
\begin{equation}
\begin{split}
 \RD(\hat{\theta}_{\JSJS},\hat{\theta}_{\phi};p)
&= \{4I(\Phi_p,1/(p-2);p)-1\}\int_0^\infty h(w) g_p(w)\rd w\\
&\geq \{4\tilde{I}(p)-1\}\int_0^\infty h(w) g_p(w)\rd w.
\end{split} 
\end{equation}
When $\tilde{I}(p)$ is increasing in $p$, $ 4\tilde{I}(p_*)-1\geq 0$ 
implies $4\tilde{I}(p)-1\geq 0$ for $p> p_*$.
Then the following result holds.
\begin{proposition}\label{pro.p.p.p}
 Suppose $ \Phi(w)$ depend on $p$ and 
$I(\Phi_p,1/(p-2);p)> \tilde{I}(p)$ where $\tilde{I}(p)$ is increasing in $p$.
Further suppose $4\tilde{I}(p_*)-1\geq 0$ for some $ p_*\geq 3$.
Then 
\begin{equation}
 \RD(\hat{\theta}_{\JSJS},\hat{\theta}_{\phi};p)\geq 0  \text{ for all }  p\geq p_*.
\end{equation}
\end{proposition}

\section{Examples}
\label{sec:examples}
%\begin{example}
%\label{ex:Phi12}
This section provides specific examples of estimators 
that fulfill the uniformity of Assumption \ref{mt.2},
discussed in Section \ref{sec:unif}.

As a function $\Phi(w)$ independent of $p$, let us consider
\begin{equation}\label{Phi.11}
 \Phi_1(w;b)=\frac{1}{b}\frac{w}{w+1},
\end{equation}
for $0<b<4$ (See \eqref{kyoku.1} in Section \ref{rem:zenkin}). 
Then the induced $\phi$ is
\begin{equation}
 \phi_1(w)=p-2- \frac{1}{\log(w+1)/b+C}.
\end{equation}
We consider two cases $C=1$ and $C=1/(p-2)$,
so that we can apply Propositions \ref{prop:1} and \ref{prop:2} respectively.
When $C=1/(p-2)$, $\beta+2$ in Proposition \ref{prop:2} is equal to $3$ since
\begin{equation}
 \frac{\Phi_1(w)}{\int_0^w t^{-1}\Phi_1(t)\rd t }
=\frac{w}{(w+1)\log(w+1)}\leq \frac{1}{w+1}\leq 1,\quad\forall \ w\geq 0.
\end{equation}
Table \ref{tab:example.0} shows which pairs of the dimension $p$ and parameters $(b,C)$
lead the shrinkage estimator with $\phi_1(w)$ to improve upon the James--Stein estimator. 
A star ($\star$) indicates pairs for which Assumption \ref{mt.2} is individually fulfilled, 
while a bullet ($\bullet$) denotes pairs for which 
Assumption \ref{mt.2} is uniformly fulfilled, 
by Propositions \ref{prop:1} and \ref{prop:2}.
A minus ($-$) indicates that Assumption \ref{mt.2} is not satisfied for the corresponding pair.

Next we consider 
\begin{equation}\label{Phi.22}
 \Phi_2(w;b,\gamma)=\left(\frac{w}{b}\right)^\gamma\quad\text{for }b>0\text{ and }\gamma>0,
\end{equation}
%for $b>0$ and $\gamma>0$, 
which does not depend on $p$.
Then the induced $\phi$ is
\begin{equation}
 \phi_2(w)=p-2- \frac{1}{w^\gamma/(b^\gamma\gamma)+C}.
\end{equation}
We treat two cases $C=1$ and $C=1/(p-2)$,
where we can apply Propositions \ref{prop:1} and \ref{prop:2}, respectively.
For $C=1$, Table \ref{tab:example.1} shows which pairs of the dimension $p$, parameters $b$ and $\gamma$
lead the shrinkage estimator with $\phi_2(w)$ to improve upon the James--Stein estimator. 
The three symbols, $\star$, $\bullet$ and $-$, have been already described.
For $C=1/(p-2)$,
$\beta$ in Proposition \ref{prop:2} is equal to $\gamma$ since
\begin{equation}
 \frac{\Phi_2(w)}{\int_0^w t^{-1}\Phi_2(t)\rd t }
=\frac{w^\gamma}{\int_0^w t^{\gamma-1}\rd t}=\gamma ,\quad\forall \ w\geq 0.
\end{equation}
Table \ref{tab:example.2} shows which pairs of the dimension $p$, parameters $b$ and $\gamma$
lead the shrinkage estimator with $\phi_2(w)$ to improve upon the James--Stein estimator. 
Unlike the case $C=1/(p-2)$ with $\Phi_1(w)$, $\beta+2$ in Proposition \ref{prop:2}
is equal to $\gamma+2$.
For example, when $b=1$ and $\gamma=5$, we have $\beta+2=\gamma+2=7$.
Then we numerically check that 
$\RD(\hat{\theta}_{\JSJS},\hat{\theta}_{\phi};p_*)\geq 0$ for $3\leq p\leq 7$
individually. 
On the other hand, $\RD(\hat{\theta}_{\JSJS},\hat{\theta}_{\phi};p_*)\geq 0$ for $p\geq 8$
follows from Proposition \ref{prop:2}.

%\begin{equation}
% \RD(\hat{\theta}_{\JSJS},\hat{\theta}_{\phi};p_*)\geq 0  \text{ for some } p_*\geq \beta+2.
%\end{equation}

%\begin{equation}
% \phi(w)=p-2- \frac{b(p-2)}{(p-2)w+b}.
%\end{equation}
%\item
%Let  
%\begin{equation}
% \Phi(w)=\frac{w^2}{b}
%\end{equation}
%for $b>0$. Then we have
%\begin{equation}
% \phi(w)=p-2- \frac{2b(p-2)}{(p-2)w^2+2b}.
%\end{equation}

\begin{table}[htbp]
  \centering
  \caption{$\Phi_1$ with $C=1, 1/(p-2)$}
  \label{tab:example.0}
  \begin{tabular}{cccccccccccccccc}
    \toprule
& \multicolumn{7}{l}{$C=1$} &\quad\quad &\multicolumn{7}{l}{$C=1/(p-2)$} \\
$p\backslash b$ & $\frac{1}{2}$ & $1$ & $\frac{3}{2}$ & $2$ & $\frac{5}{2}$ & $3$ & $\frac{7}{2}$ &\quad& $\frac{1}{2}$ & $1$ & $\frac{3}{2}$ & $2$ & $\frac{5}{2}$ & $3$ & $\frac{7}{2}$ \\[2pt]
 \cmidrule(lr){2-8}\cmidrule(lr){10-16}\\[-8pt]
3 & $\star$ & - & - & - & - & - & - & & $\star$  & - &  - &  - &  - &  - &  - \\
4 & $\bullet$ & $\star$ & $\star$ & - & - & - & - & & $\bullet$ & $\star$ &  - &  - &  - &  - &  - \\
5 & $\bullet$ & $\bullet$ & $\bullet$ & $\star$ & - & - & - & & $\bullet$ & $\bullet$ &  $\star$ &  $\star$ &  - &  - &  - \\
6 & $\bullet$ & $\bullet$ & $\bullet$ & $\bullet$ & $\star$ & - & - & & $\bullet$ & $\bullet$ &  $\bullet$ &  $\bullet$ &  - &  - &  - \\
7 & $\bullet$ & $\bullet$ & $\bullet$ & $\bullet$ & $\bullet$ & $\star$ & - & & $\bullet$ & $\bullet$ &  $\bullet$ &  $\bullet$ &  $\star$ &  - &  - \\
8 & $\bullet$ & $\bullet$ & $\bullet$ & $\bullet$ & $\bullet$ & $\bullet$ & - & & $\bullet$ & $\bullet$ &  $\bullet$ &  $\bullet$ &  $\bullet$ &  $\star$ &  - \\
9 & $\bullet$ & $\bullet$ & $\bullet$ & $\bullet$ & $\bullet$ & $\bullet$ & - & & $\bullet$ & $\bullet$ &  $\bullet$ &  $\bullet$ &  $\bullet$ &  $\bullet$ &  - \\
10 & $\bullet$ & $\bullet$ & $\bullet$ & $\bullet$ & $\bullet$ & $\bullet$ & - & & $\bullet$ & $\bullet$ &  $\bullet$ &  $\bullet$ &  $\bullet$ &  $\bullet$ &  - \\
 \bottomrule
  \end{tabular}
\end{table}

\begin{table}[htbp]
  \centering
  \caption{$\Phi_2$ with $C=1$}
  \label{tab:example.1}
  \begin{tabular}{cccccccccccccccccccccccccccccccccccc}
    \toprule
& \multicolumn{7}{l}{$b=1$} &\multicolumn{7}{l}{$b=3$} &\multicolumn{7}{l}{$b=5$} &\multicolumn{7}{l}{$b=7$} &\multicolumn{7}{l}{$b=9$} \\
$p\backslash\gamma$ & $\frac{1}{4}$ & $\frac{1}{2}$ & $1$ & $2$ & $3$ & $4$ & $5$ & $\frac{1}{4}$ & $\frac{1}{2}$ & $1$ & $2$ & $3$ & $4$ & $5$ & $\frac{1}{4}$ & $\frac{1}{2}$ & $1$ & $2$ & $3$ & $4$ & $5$ & $\frac{1}{4}$ & $\frac{1}{2}$ & $1$ & $2$ & $3$ & $4$ & $5$ & $\frac{1}{4}$ & $\frac{1}{2}$ & $1$ & $2$ & $3$ & $4$ & $5$ \\[2pt]
%    \midrule
 \cmidrule(lr){2-8}\cmidrule(lr){9-15}\cmidrule(lr){16-22}\cmidrule(lr){23-29}\cmidrule(lr){30-36}\\[-8pt]
3& $\star$ &  $\star$ &  $\star$ &  $\star$ &  $\star$ &  $\star$ &  $\star$ &  $\star$ &  $\star$ &   - &   - &   - &   - &   - &   $\star$ &   - &   - &   - &   - &   - &   - &   $\star$ &   - &   - &   - &   - &   - &   - &   $\star$ &   - &   - &   - &   - &   - &   - \\
4 &  $\bullet$ &  $\bullet$ &  $\bullet$ &  $\bullet$ &  $\bullet$ &  $\bullet$ &  $\bullet$ &  $\bullet$ &  $\bullet$ &   $\star$ &   $\star$ &   $\star$ &   $\star$ &   $\star$ &   $\bullet$ &   $\star$ &   $\star$ &   - &   - &   - &   - &   $\bullet$ &   $\star$ &   - &   - &   - &   - &   - &   $\bullet$ &   $\star$ &   - &   - &   - &   - &   - \\
5 &  $\bullet$ &  $\bullet$ &  $\bullet$ &  $\bullet$ &  $\bullet$ &  $\bullet$ &  $\bullet$ &  $\bullet$ &  $\bullet$ &   $\bullet$ &   $\bullet$ &   $\bullet$ &   $\bullet$ &   $\bullet$ &   $\bullet$ &   $\bullet$ &   $\bullet$ &   $\star$ &   $\star$ &   $\star$ &   $\star$ &   $\bullet$ &   $\bullet$ &   $\star$ &   - &   - &   - &   - &   $\bullet$ &   $\bullet$ &   $\star$ &   - &   - &   - &   - \\
6 &  $\bullet$ &  $\bullet$ &  $\bullet$ &  $\bullet$ &  $\bullet$ &  $\bullet$ &  $\bullet$ &  $\bullet$ &  $\bullet$ &   $\bullet$ &   $\bullet$ &   $\bullet$ &   $\bullet$ &   $\bullet$ &   $\bullet$ &   $\bullet$ &   $\bullet$ &   $\bullet$ &   $\bullet$ &   $\bullet$ &   $\bullet$ &   $\bullet$ &   $\bullet$ &   $\bullet$ &   $\star$ &   $\star$ &   - &   - &   $\bullet$ &   $\bullet$ &   $\bullet$ &   - &   - &   - &   - \\
7 &  $\bullet$ &  $\bullet$ &  $\bullet$ &  $\bullet$ &  $\bullet$ &  $\bullet$ &  $\bullet$ &  $\bullet$ &  $\bullet$ &   $\bullet$ &   $\bullet$ &   $\bullet$ &   $\bullet$ &   $\bullet$ &   $\bullet$ &   $\bullet$ &   $\bullet$ &   $\bullet$ &   $\bullet$ &   $\bullet$ &   $\bullet$ &   $\bullet$ &   $\bullet$ &   $\bullet$ &   $\bullet$ &   $\bullet$ &   $\star$ &   $\star$ &   $\bullet$ &   $\bullet$ &   $\bullet$ &   $\star$ &   - &   - &   - \\
8 &  $\bullet$ &  $\bullet$ &  $\bullet$ &  $\bullet$ &  $\bullet$ &  $\bullet$ &  $\bullet$ &  $\bullet$ &  $\bullet$ &   $\bullet$ &   $\bullet$ &   $\bullet$ &   $\bullet$ &   $\bullet$ &   $\bullet$ &   $\bullet$ &   $\bullet$ &   $\bullet$ &   $\bullet$ &   $\bullet$ &   $\bullet$ &   $\bullet$ &   $\bullet$ &   $\bullet$ &   $\bullet$ &   $\bullet$ &   $\bullet$ &   $\bullet$ &   $\bullet$ &   $\bullet$ &   $\bullet$ &   $\bullet$ &   $\star$ &   $\star$ &   $\star$\\
9 &  $\bullet$ &  $\bullet$ &  $\bullet$ &  $\bullet$ &  $\bullet$ &  $\bullet$ &  $\bullet$ &  $\bullet$ &  $\bullet$ &   $\bullet$ &   $\bullet$ &   $\bullet$ &   $\bullet$ &   $\bullet$ &   $\bullet$ &   $\bullet$ &   $\bullet$ &   $\bullet$ &   $\bullet$ &   $\bullet$ &   $\bullet$ &   $\bullet$ &   $\bullet$ &   $\bullet$ &   $\bullet$ &   $\bullet$ &   $\bullet$ &   $\bullet$ &   $\bullet$ &   $\bullet$ &   $\bullet$ &   $\bullet$ &   $\bullet$ &   $\bullet$ &   $\bullet$\\
10 &  $\bullet$ &  $\bullet$ &  $\bullet$ &  $\bullet$ &  $\bullet$ &  $\bullet$ &  $\bullet$ &  $\bullet$ & $\bullet$ &  $\bullet$ &  $\bullet$ &  $\bullet$ &  $\bullet$ & $\bullet$ & $\bullet$ & $\bullet$ &  $\bullet$ & $\bullet$ & $\bullet$ &  $\bullet$ & $\bullet$ & $\bullet$ & $\bullet$  & $\bullet$ & $\bullet$ & $\bullet$ & $\bullet$ & $\bullet$ & $\bullet$ & $\bullet$ &  $\bullet$ & $\bullet$ & $\bullet$ & $\bullet$ & $\bullet$ \\
 \bottomrule
  \end{tabular}
\end{table}

\begin{table}[htbp]
  \centering
  \caption{$\Phi_2$ with $C=1/(p-2)$}
  \label{tab:example.2}
  \begin{tabular}{cccccccccccccccccccccccccccccccccccc}
    \toprule
& \multicolumn{7}{l}{$b=1$} &\multicolumn{7}{l}{$b=3$} &\multicolumn{7}{l}{$b=5$} &\multicolumn{7}{l}{$b=7$} &\multicolumn{7}{l}{$b=9$} \\
$p\backslash\gamma$ & $\frac{1}{4}$ & $\frac{1}{2}$ & $1$ & $2$ & $3$ & $4$ & $5$ & $\frac{1}{4}$ & $\frac{1}{2}$ & $1$ & $2$ & $3$ & $4$ & $5$ & $\frac{1}{4}$ & $\frac{1}{2}$ & $1$ & $2$ & $3$ & $4$ & $5$ & $\frac{1}{4}$ & $\frac{1}{2}$ & $1$ & $2$ & $3$ & $4$ & $5$ & $\frac{1}{4}$ & $\frac{1}{2}$ & $1$ & $2$ & $3$ & $4$ & $5$ \\[2pt]
 \cmidrule(lr){2-8}\cmidrule(lr){9-15}\cmidrule(lr){16-22}\cmidrule(lr){23-29}\cmidrule(lr){30-36}\\[-8pt]
3 & $\star$ & $\star$ & $\star$ & $\star$ & $\star$ & $\star$ & $\star$ & $\star$ & $\star$ &  - &  - &  - &  - &  - &  $\star$ &  - &  - &  - &  - &  - &  - &  $\star$ &  - &  - &  - &  - &  - &  - &  $\star$ &  - &  - &  - &  - &  - &  - \\
   4 & $\bullet$ & $\bullet$ & $\bullet$ & $\star$ & $\star$ & $\star$ & $\star$ & $\bullet$ & $\bullet$ &  $\star$ &  - &  - &  - &  - &  $\bullet$ &  $\star$ &  - &  - &  - &  - &  - &  $\bullet$ &  $\star$ &  - &  - &  - &  - &  - &  $\bullet$ &  $\star$ &  - &  - &  - &  - &  - \\
5 & $\bullet$ & $\bullet$ & $\bullet$ & $\bullet$ & $\star$ & $\star$ & $\star$ & $\bullet$ & $\bullet$ &  $\bullet$ &  $\bullet$ &  $\star$ &  $\star$ &  - &  $\bullet$ &  $\bullet$ &  $\star$ &  - &  - &  - &  - &  $\bullet$ &  $\bullet$ &  - &  - &  - &  - &  - &  $\bullet$ &  $\bullet$ &  - &  - &  - &  - &  - \\
6 & $\bullet$ & $\bullet$ & $\bullet$ & $\bullet$ & $\bullet$ & $\star$ & $\star$ & $\bullet$ & $\bullet$ &  $\bullet$ &  $\bullet$ &  $\bullet$ &  $\star$ &  $\star$ &  $\bullet$ &  $\bullet$ &  $\star$ &  - &  - &  - &  - &  $\bullet$ &  $\bullet$ &  $\star$ &  - &  - &  - &  - &  $\bullet$ &  $\bullet$ &  $\star$ &  - &  - &  - &  - \\
7 & $\bullet$ & $\bullet$ & $\bullet$ & $\bullet$ & $\bullet$ & $\bullet$ & $\star$ & $\bullet$ & $\bullet$ &  $\bullet$ &  $\bullet$ &  $\bullet$ &  $\bullet$ &  $\star$ &  $\bullet$ &  $\bullet$ &  $\bullet$ &  $\star$ &  $\star$ &  - &  - &  $\bullet$ &  $\bullet$ &  $\bullet$ &  - &  - &  - &  - &  $\bullet$ &  $\bullet$ &  $\bullet$ &  - &  - &  - &  - \\
8 & $\bullet$ & $\bullet$ & $\bullet$ & $\bullet$ & $\bullet$ & $\bullet$ & $\bullet$ & $\bullet$ & $\bullet$ &  $\bullet$ &  $\bullet$ &  $\bullet$ &  $\bullet$ &  $\bullet$ &  $\bullet$ &  $\bullet$ &  $\bullet$ &  $\bullet$ &  $\bullet$ &  $\star$ &  $\star$ &  $\bullet$ &  $\bullet$ &  $\bullet$ &  $\star$ &  - &  - &  - &  $\bullet$ &  $\bullet$ &  $\bullet$ &  - &  - &  - &  - \\
9 & $\bullet$ & $\bullet$ & $\bullet$ & $\bullet$ & $\bullet$ & $\bullet$ & $\bullet$ & $\bullet$ & $\bullet$ &  $\bullet$ &  $\bullet$ &  $\bullet$ &  $\bullet$ &  $\bullet$ &  $\bullet$ &  $\bullet$ &  $\bullet$ &  $\bullet$ &  $\bullet$ &  $\bullet$ &  $\bullet$ &  $\bullet$ &  $\bullet$ &  $\bullet$ &  $\bullet$ &  $\star$ &  - &  - &  $\bullet$ &  $\bullet$ &  $\bullet$ &  $\star$ &  - &  - &  - \\
10 & $\bullet$ & $\bullet$ & $\bullet$ & $\bullet$ & $\bullet$ & $\bullet$ & $\bullet$ & $\bullet$ & $\bullet$ &  $\bullet$ &  $\bullet$ &  $\bullet$ &  $\bullet$ &  $\bullet$ &  $\bullet$ &  $\bullet$ &  $\bullet$ &  $\bullet$ &  $\bullet$ &  $\bullet$ &  $\bullet$ &  $\bullet$ &  $\bullet$ &  $\bullet$ &  $\bullet$ &  $\bullet$ &  $\star$ &  $\star$ &  $\bullet$ &  $\bullet$ &  $\bullet$ &  $\bullet$ &  - &  - &  - \\
 \bottomrule
  \end{tabular}
\end{table}

Finally let
\begin{equation}\label{Phi.33}
 \Phi_3(w)=\frac{w}{a(p-2)},\quad\text{for }a>0,
\end{equation}
which depend on $p$. 
Then we have 
\begin{equation}
 \phi_3(w)=p-2- (p-2)\frac{a}{w+a} =(p-2)\frac{w}{w+a}
\end{equation}
and the corresponding estimator
\begin{equation}
\hat{\theta}_{a,p-2}= \left(1-\frac{p-2}{\|X\|^2+a}\right)X
\end{equation}
which is the member of \citeapos{Stein-1956} initial class \eqref{theta.a.b}.
In this case,
\begin{equation}
 I\Bigl(\frac{w}{a(p-2)},\frac{1}{p-2};p\Bigr)
%\frac{\int_0^\infty C(w) \{c(w)\}^2 w^{p/2-2}e^{-w/2}\rd w}
%{\int_0^\infty \{c(w;a)\}^2 w^{p/2-2}e^{-w/2}\rd w}
=\frac{1}{a(p-2)}
\frac{\int_0^\infty w (w+b)^{-2} w^{p/2-2}e^{-w/2}\rd w}
{\int_0^\infty (w+b)^{-2} w^{p/2-2}e^{-w/2}\rd w}.
\end{equation}
Suppose $p\geq 7$. Then, it follows from the covariance inequality that
\begin{equation}
\begin{split}
& I\Bigl(\frac{w}{a(p-2)},\frac{1}{p-2};p\Bigr)=\frac{1}{a(p-2)}\frac{\int_0^\infty w \{w/(w+a)\}^2 w^{p/2-4}e^{-w/2}\rd w}
{\int_0^\infty \{w/(w+a)\}^2 w^{p/2-4}e^{-w/2}\rd w}\\
&\geq  
\frac{1}{a(p-2)}\frac{\int_0^\infty  w^{p/2-3}e^{-w/2}\rd w}
{\int_0^\infty w^{p/2-4}e^{-w/2}\rd w}
=\frac{1}{a(p-2)}\frac{2^{p/2-2}\Gamma(p/2-2)}{2^{p/2-3}\Gamma(p/2-3)}=\frac{p-6}{a(p-2)},
\end{split}
\end{equation}
which is increasing in $p$. Hence, for $0<a<4$ and $p\geq 2(12-a)/(4-a)$, we have
\begin{equation}
I\Bigl(\frac{w}{a(p-2)},\frac{1}{p-2};p\Bigr)\geq \frac{p-6}{a(p-2)} =\frac{1}{4}.
\end{equation}
Then from Proposition \ref{pro.p.p.p}
the corresponding estimator dominates the James--Stein estimator.

Nonetheless, it is not ruled out that it may outperform the James--Stein estimator 
in cases where $ p<(12-a)/(4-a)$. Let $a=1$ for example.
By numerical investigation, we can check
\begin{equation}
 I\Bigl(\frac{w}{p-2},\frac{1}{p-2};p\Bigr)\geq \frac{1}{4} 
\end{equation}
individually for $3\leq p\leq 7$ and hence we can conclude that the corresponding estimator 
\begin{equation}
 \left(1-\frac{p-2}{\|X\|^2+1}\right)X
\end{equation}
dominates the James--Stein estimator for $p\geq 3$.

\section{Asymptotic risk gain}
%\begin{remark}
\label{rem:zenkin}
Here we are interested in asymptotic risk gain
$R(\theta,\hat{\theta}_{\JSJS})- R(\theta,\hat{\theta}_{\phi})$ as $\|\theta\|\to\infty$.
\cite{Maruyama-Takemura-2024} investigated the difference in risks 
between $\hat{\theta}_{\JSJS}$ and $\hat{\theta}_{\JSJS}^+$ as follows,
\begin{equation} 
\lim_{\|\theta\|\to\infty}
\frac{\|\theta\|^{(p+1)/2}e^{\|\theta\|^2/2}}{e^{\|\theta\|\sqrt{p-2}}} 
 \bigl\{R(\theta,\hat{\theta}_{\JSJS})-R(\theta,\hat{\theta}_{\JSJS}^+)\bigr\}
=4\frac{(p-2)^{(p-1)/4}}{\sqrt{2\pi}\exp(p/2-1)},
\end{equation}
which implies that the asymptotic risk improvement of the positive--part estimator 
is negligible.
In this section, we show that the shrinkage estimators induced by
$\Phi_1$ in \eqref{Phi.11}, $\Phi_2$ in \eqref{Phi.22} and $\Phi_3$ in \eqref{Phi.33}
achieve better asymptotic performance.

The functions
$\Phi_1$ in \eqref{Phi.11}, $\Phi_2$ in \eqref{Phi.22}, and $\Phi_3$ in \eqref{Phi.33},
along with the corresponding induced functions $\phi$ are all
regularly varying.
Recall \cite{Berger-1975} and \cite{Maruyama-Takemura-2008}
showed that
\begin{equation}\label{Berger-Maruyama-Takemura}
 E[G(\|X\|^2)]\approx G(\|\theta\|^2)\text{ for }\|\theta\|^2\to\infty,
\end{equation}
for some regularly varying $G(w)$, which can be applied for 
asymptotic evaluation of the difference in risks. 

For $\Phi_1$ in \eqref{Phi.11}, 
the corresponding difference in unbiased estimator of risks
satisfies
\begin{equation}\label{alpha.0}
\lim_{w\to\infty}
w\{\log w\}^{2} \{\hat{R}_{\JSJS}(w)-\hat{R}_\phi(w)\} =-b^2+4b,
\end{equation}
and hence, by \eqref{Berger-Maruyama-Takemura},
\begin{equation}\label{kyoku.1}
\lim_{\|\theta\|\to\infty} \|\theta\|^2\{\log \|\theta\|^2\}^{2}
\{R(\theta,\hat{\theta}_{\JSJS})- R(\theta,\hat{\theta}_{\phi})\}
=-b^2+4b.
\end{equation}
Hence $\hat{\theta}_{\phi}$ asymptotically dominates the James--Stein estimator if
\begin{equation}\label{alpha}
 0<b<4.
\end{equation}
The asymptotic risk gain is maximized at $b=2$ and
\begin{equation}\label{kyoku.1.5}
\lim_{\|\theta\|\to\infty} \|\theta\|^2\{\log \|\theta\|^2\}^{2}
\{R(\theta,\hat{\theta}_{\JSJS})- R(\theta,\hat{\theta}_{\phi})\}
=4.
\end{equation}
For $\Phi_2$ in \eqref{Phi.22}, the corresponding difference in unbiased estimator of risks
satisfies
\begin{equation}\label{alpha.1}
\lim_{w\to\infty}
w^{1+\gamma} \{\hat{R}_{\JSJS}(w)-\hat{R}_\phi(w)\} =4\gamma^2b^\gamma,
\end{equation}
and hence
\begin{equation}\label{kyoku.2}
\lim_{\|\theta\|\to\infty} \|\theta\|^{2(1+\gamma)}
\{R(\theta,\hat{\theta}_{\JSJS})- R(\theta,\hat{\theta}_{\phi})\}
=4\gamma^2 b^\gamma.
\end{equation}
which implies that 
$\hat{\theta}_{\phi}$ asymptotically dominates the James--Stein estimator for $\gamma>0$
and $b>0$.
The function $\Phi_3$ in \eqref{Phi.33} corresponds to $ \Phi_2$ with $\gamma=1$ and
$b=a(p-2)$. Then, for $\Phi_3$, we have 
\begin{equation}\label{kyoku.3}
\lim_{\|\theta\|\to\infty} \|\theta\|^{4}
\{R(\theta,\hat{\theta}_{\JSJS})- R(\theta,\hat{\theta}_{\phi})\}=4a(p-2).
\end{equation}
From \eqref{kyoku.1}, \eqref{kyoku.1.5}, \eqref{kyoku.2}, and \eqref{kyoku.3}, 
the optimal asymptotic risk gain is attained by $\hat{\theta}_{\phi 1}$
with 
\begin{equation}\label{phi.log.00}
 \phi_1(w)=p-2- \frac{2}{\log(w+1)+2C},
\end{equation}
among the class of $\phi(w)$ induced by $\Phi_1$, $\Phi_2$ and $\Phi_3$.
Furthermore, $\phi_1(w)$ given by \eqref{phi.log.00}
remains optimal even within a broader class of functions
\begin{equation}
 \left\{\phi(w)\mid \phi(w)\approx p-2-\frac{\alpha}{(\log w)^\beta w^\gamma}\text{ as }w\to\infty\right\}.
\end{equation}
%\end{remark}

\section{Concluding Remarks}
\label{sec:CR}
This paper introduces a new framework for constructing shrinkage estimators that dominate the James--Stein estimator under quadratic loss. 
By leveraging a monotonicity condition on a transformed shrinkage function, 
we derive a general class of estimators that satisfy dominance 
over the James--Stein estimator. 
The approach allows for polynomial or logarithmic convergence to the optimal shrinkage factor.
Furthermore, the paper provides sufficient conditions for uniform dominance across dimensions, enabling practical application without dimension--specific verification. 
Several examples, including estimators with improved asymptotic risk properties, are presented to illustrate the theory. 

We believe that the results of this paper contribute to a deeper understanding 
of shrinkage estimation.
On the other hand, all estimators proposed in this paper, including
\begin{equation}
\left(1-\frac{p-2}{\|X\|^2+1}\right)X,\quad
\left(1-\frac{1}{\|X\|^2}\left\{p-2-\frac{2}{\log (\|X\|^2+1)+2/(p-2)}\right\} \right)
X
\end{equation}
are inadmissible since their corresponding positive--part estimators dominate them.
Consequently, exploring admissible generalized Bayes estimators that fulfill our sufficient conditions represents a promising direction for future research.

\end{document}